\documentclass[a4paper,12pt,twoside]{article}

\usepackage{charter}

\title{On the Hausdorff dimension of a 2-dimensional Weierstrass curve}
\author{
Peter Imkeller$^1$ \footnote{P.~Imkeller was supported in part by DFG Research Unit FOR 2402.}\\
        \small imkeller@mathematik.hu-berlin.de
\and
Gon\c calo dos Reis$^{2,3}$ \footnote{G.~dos Reis acknowledges support from the \emph{Funda{\c c}$\tilde{\text{a}}$o para a Ci$\hat{e}$ncia e a Tecnologia} (Portuguese Foundation for Science and Technology) through the project UID/MAT/00297/2013 (Centro de Matem\'atica e Aplica\c c$\tilde{\text{o}}$es CMA/FCT/UNL).} \\ 
        \small  G.dosReis@ed.ac.uk
}

\date{ 
\normalsize $^1$Humboldt Universit\"at zu Berlin\\%
    $^2$University of Edinburgh\\
		$^3$Centro de Matem\'atica e Aplica\c c$\tilde{\text{o}}$es, (CMA), FCT, UNL\\
		[2ex]%
		\currenttime, \ddmmyyyydate\today 
		\qquad{(File \tt \jobname.tex})
		}

\usepackage{amssymb}
\usepackage[english]{babel}

\usepackage{amsmath}
\usepackage{amsthm,bbm}

\usepackage{graphicx}

\usepackage[dvipsnames]{xcolor}

\numberwithin{equation}{section}
 \usepackage[nodayofweek]{datetime}

\newcommand{\dt}{\mathrm{d}t}
\newcommand{\dz}{\mathrm{d}z}

\def\om{\omega}
\def\Om{\Omega}

\def\un{\infty}
\def\eh{\frac{1}{2}}
\def\nh{\frac{n}{2}}

\def\zp{2 \pi}

\def\cosi{\left( \begin{array}{c}
\cos \\ {\sin}
\end{array}\right)}
\def\sico{\left( \begin{array}{c}
-\sin\\ \cos
\end{array}\right)}

\def\nn{\nonumber}

\def\bx{\overline{x}}
\def\bxi{\overline{\xi}}

\newtheorem{thm}{Theorem}[section]
\newtheorem{pr}[thm]{Proposition}
\newtheorem{co}[thm]{Corollary}
\newtheorem{lem}[thm]{Lemma}

\newcommand{\C}{\mathbb{C}}

\newcommand{\N}{\mathbb{N}}
\newcommand{\R}{\mathbb{R}}
\newcommand{\Z}{\mathbb{Z}}

\newcommand{\be}{\begin{equation}}
\newcommand{\ee}{\end{equation}}
\newcommand{\bdm}{\begin{displaymath}}
\newcommand{\edm}{\end{displaymath}}
\newcommand{\bean}{\begin{eqnarray}}
\newcommand{\eean}{\end{eqnarray}}
\newcommand{\bea}{\begin{eqnarray*}}
\newcommand{\eea}{\end{eqnarray*}}

\newcommand{\ud}{\mathrm{d}}
\newcommand{\1}{\mathbbm{1}}

\newcommand{\cB}{\mathcal{B}}

\begin{document}
\selectlanguage{english}
\maketitle

\begin{abstract}
We compute the Hausdorff dimension of a two-dimensional Weierstrass function, related to lacunary (Hadamard gap) power series, that has no L\'evy area. This is done by interpreting it as a pullback attractor of a dynamical system based on the Baker transformation. A lower bound for the Hausdorff dimension is obtained by investigating the pushforward of the Lebesgue measure on the graph along scaled neighborhoods of stable fibers of the underlying dynamical system following the graph. Scaling ideas are crucial. They become accessible by self similarity properties of a mapping whose increments coincide with vertical distances on the stable fibers.

\end{abstract}



\section{Introduction}

The study of the 2-dimensional Weierstrass function considered in this paper had its starting point in a Fourier analytic approach of rough path analysis or rough integration theory laid out in \cite{gubinelliimkellerperkowski2015} and \cite{gubinelliimkellerperkowski2016}. In \cite{gubinelliimkellerperkowski2016}, the construction of a Stratonovich type integral of a rough function $f$ with respect to another rough function $g$ is related to the notion of paracontrol of $f$ by $g$. This Fourier analytic concept generalizes the original control concept introduced by Gubinelli \cite{gubinelli2004}. In search of a good example of a pair of functions not controlling each other, in \cite{imkellerproemel15} the authors came up with the pair of Weierstrass functions defined by 
\begin{align*}
W=(W_1, W_2) = \sum_{n=0}^\infty 2^{-\nh} \cosi \big(\zp 2^n x\big),\quad x\in[0,1].
\end{align*}
The first component $W_1$ fluctuates on all dyadic scales in a cosinusoidal manner while the second one in a sinusoidal way. Hence when the first one has minimal increments, the second one has maximal ones, and vice versa. This can be seen to mathematically rigorously underpin the fact that they are mutually not controlled (see \cite[Example 2.8]{imkellerproemel15}). It also implies that the L\'evy areas of the approximating finite sums of the representing series do not converge. This geometric pathology motivated us to look for further geometric properties of the pair. The question arose: if it is so difficult to define an integral of one component with respect to the other, and L\'evy's area fails to exist, is the curve possibly space filling, at least on a nontrivial portion of its graph? This gave rise to the main goal of this paper: to investigate at least one characteristic of singularity of the curve, the Hausdorff dimension of its graph. We shall show that it equals 2, substantiating our conjecture that the curve is thick in space. Functions with a uniform fractal structure on the domain such as our Weierstrass functions occur everywhere in nature. They are widely applicable in physics because their graphs play an important role as invariant sets in dynamical systems. Following \cite{carvalho2011}, results on sets such as the graphs of $W$ are of independent interest in the investigation of SPDE in dimensions higher than one that are defined on fractal sets.

Figure \ref{fig:idrRangeW-v01} and \ref{fig:idrRangeW-v01-3D} show, respectively, the Range of $W$ over $[0,1]$ (subset of $\R^2$) and the Graph of $W$ over $[0,1]$ (subset of $\R^3$).
\begin{figure}[htbp]
	\centering
		\includegraphics[width=0.5\textwidth]{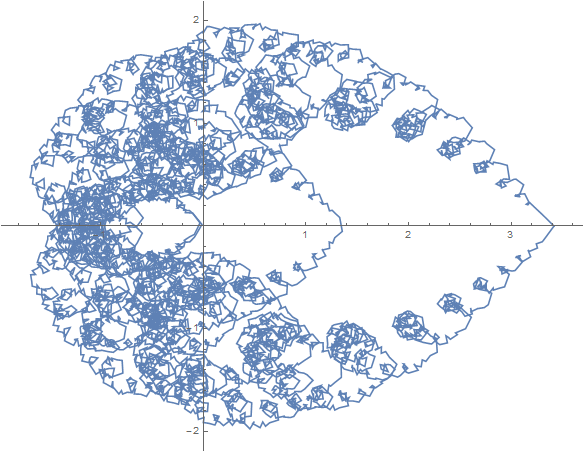}
	\caption{Parametric plot (Range) of $W$ for $x\in[0,1]$; $\{W(x):x\in[0,1]\}\subset \R^2$.}
	\label{fig:idrRangeW-v01}
\end{figure}

\begin{figure}[htbp]
	\centering
		\includegraphics[width=0.5\textwidth]{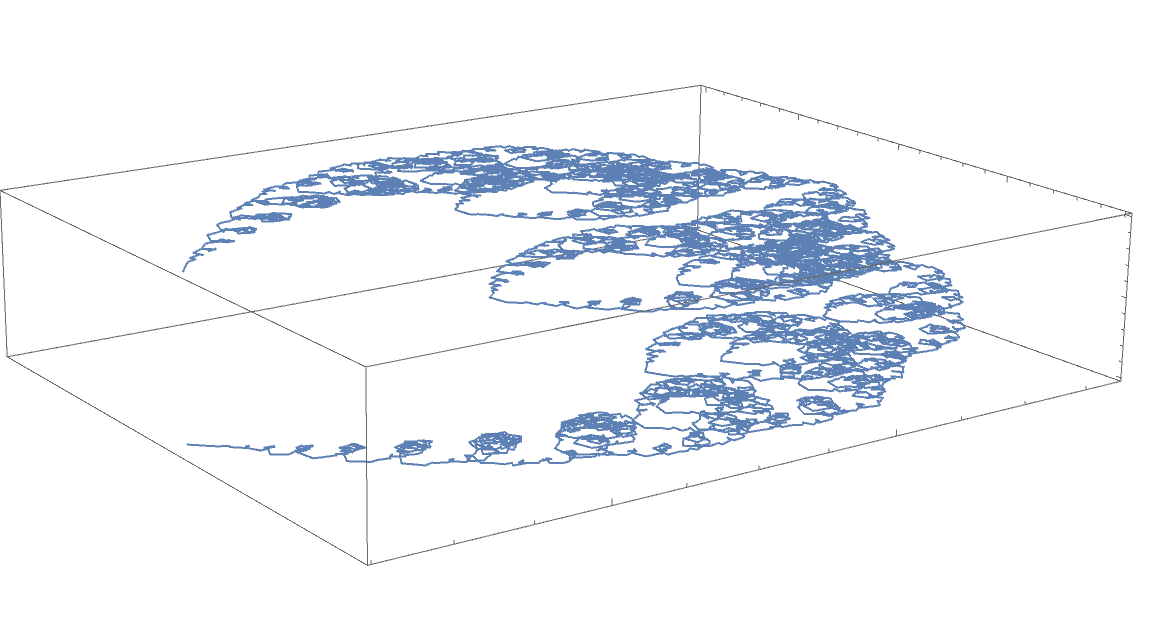}
	\caption{Graph of $W$; $\{\big(x,W(x)\big):x\in[0,1]\}\subset \R^3$.}
	\label{fig:idrRangeW-v01-3D}
\end{figure}

The map $W$ can also be understood as the real and imaginary part of a certain complex function, namely,
\[
\widehat W(z) = \sum_{n=0}^\infty 2^{-\nh} z^{2^n},\quad \quad z\in\C,\ |z|<1.
\]
This analytical interpretation was the original idea of Hardy \cite{hardy1916} to prove that the components of $W$ are nowhere differentiable. The map $\widehat W$ is referred to as the lacunary (Hadamard gaps) complex power series (see \cite{baranski02}).
\medskip

\emph{Methodological guidelines.} It has been noticed in a number of works on one dimensional Weierstrass type curves (see \cite{hunt1998}, \cite{baranski02}, \cite{baranski12}, \cite{baranski14published}, \cite{baranski15survey}, \cite{keller17-Publication-of-2015}, \cite{shen2017-Publication-of-2015}) that the number of iterations of the expansion by a real factor present in the arguments of the terms of their expansion can be taken as a starting point in interpreting their graphs as pullback attractors of dynamical systems in which a baker transformation defines the dynamics. This observation marks, in many of the mentioned works, the point of departure for determining the Hausdorff dimension of graphs of one dimensional Weierstrass type functions. For a historical survey of this work the reader may consult \cite{baranski14published}. For our curve we use the same metric dynamical system based on a suitable baker transformation $B$ as a starting point. This is done by introducing, besides a variable $x$ that encodes expansion by the factor $2$ forward in time, an auxiliary variable $\xi$ describing contraction by the factor $\frac{1}{2}$ in turn, forward in time as well. Backward in time, the sense of expansion and contraction is interchanged. Consequently, if $\gamma = \frac{1}{\sqrt{2}}$, the $n$-th term of the series representing $W$ is given by $\gamma^{n} \cosi\big(\zp B^{-n}(\xi, x)\big)$. The action of applying a forward expansion in one step just corresponds to stepping from one term in the expansion of $W$ to the following one. This indicates that $W$ is an attractor of a dynamical system $F$ that, besides contracting the two leading variables by the factor $\gamma$, adds the first term of the series to the result. So by definition of $F$, $W$ is its attractor. Since $\frac{1}{2}$, the factor of $x$ in the forward fibre motion, is the smallest Lyapunov exponent of the linearization of $F$, there is a stable manifold related to this Lyapunov exponent. It is spanned by the vector which is given as a Weierstrass type series $S(\xi,x)$, the $n$-th term of which is given by $- \zp \gamma^n \sico\big(\zp B^n(\xi,x)\big)$, as will be explained below. The pushforward of the Lebesgue measure by $S(\cdot, x)$ for $x\in[0,1]$ fixed, is the $x$-marginal of the Sinai-Bowen-Ruelle measure of $F$. The definition of $F$ as a linear transformation added to a very smooth function may be understood as conveying the concept of \emph{self-affinity} for the Weierstrass curve. The (random) dynamical system on the metric space $[0,1]^2$ underlying our analytical interpretation gives rise to an enhancement of this property to \emph{self-similarity}, a very convenient notion widely used in the theory of the fine structure of stochastic processes. To make this important step, we define
$$H(\xi,x)
= \sum_{n\in\Z} 2^{-\frac{n}{2}} \left[\cosi \big(\zp B^{-n}_2(\xi, x)\big) - \cosi \big(\zp B^{-n}_2(\xi, 0)\big)\right],
\quad \xi, x\in[0,1].$$
In Proposition \ref{p:scaling} we assess the scaling properties of $H$, leading to self similarity of increments of $H$ with respect to Lebesgue measure.
In our key Lemma \ref{l:doubly_infinite_series} we see that $H$ and $W$ are linked by the formula
$$H(\xi,y) - H(\xi,x) = W(y) - W(x) - \int_x^y S(\xi,z) dz,$$
thus making $W$ accessible to self similarity studies. Another key observation is that our analysis provides a geometric interpretation of the increments of $H$. Define the \emph{stable fiber} through a point $(x, W(x))$ of the graph of $W$ by solutions of the initial value problem of the ODE
$$\frac{d}{d v} l_{\xi,x,w)}(v) = S(\xi, v),\quad l_{(\xi,x,w)}(x) = w,$$
where we set $w = W(x).$ Then vertical distances on different stable fibers are just given by the increments of $H$:
$$l_{(\xi, y, W(y))}(y) - l_{(\xi, x, W(x))}(y) = H(\xi,y) - H(\xi,x),\quad \xi, x, y\in[0,1].$$ 
This is the crucial observation for determining a lower bound for the Hausdorff dimension of the graph of $W$. Following Keller \cite{keller17-Publication-of-2015}, we find the sharp lower bound by investigating the local dimension of the Lebesgue measure on the graph of $W$. This is done by assessing the measure of dyadically scaled small neighborhoods of the stable fibers. By their relationship to the self similar process $H$, they become susceptible to arguments using the scaling properties of $H$. By taking full advantage of the self-similarity property we are able to present simpler proofs.

Lastly, the sharp upper bound for the Hausdorff dimension follows from classical estimates on the Box dimension (known to dominated the Hausdorff one) where we use that $W$ is H\"older continuous (see e.g. \cite[Proposition 4.14]{mortersperes10}).  More generally, \cite{baranski02} shows that the Box dimension of the graph of $W$ in \eqref{eq:DefinitionW} is $3-2\alpha$, for $\alpha\leq 1/2$. In fact, it is proven in \cite{baranski02} that for $\alpha$ sufficiently small the image of $W$ has a non-empty interior in the topology of the plane.
\smallskip 

\emph{Organization.} The manuscript is organized along the lines of reasoning described above. In Section 2, repeating \cite{baranski02}, \cite{hunt1998} or \cite{keller17-Publication-of-2015}, we explain the interpretation of our Weierstrass curve in terms of dynamical systems based on the baker transform. In Section 3, we make the step from self affinity to self similarity and explore the scaling properties of $H$. In Section 4, we use the geometric interpretation of the vertical distance between stable fibers and increments of $H$ to finally estimate a lower bound on the Hausdorff dimension of the graph of $W$. Section 5 just recapitulates simple known facts about the easier upper bounds on Hausdorff dimensions of rough graphs, leading to the main result of the paper, Theorem \ref{t:main}.

\bigskip

\textbf{Acknowledgements.} The authors thank the financial support of the \emph{International Centre for Mathematical Sciences} (ICMS) Research-in-Groups programme which allowed us to work together in Edinburgh on this manuscript.

\section{The curve as attractor of a dynamical system}\label{s:attractor}

Our aim is to investigate the Hausdorff dimension of the graph of the two-dimensional Weierstrass curve given by
\be
\label{eq:DefinitionW}
W(x) = \big( W_1(x), W_2(x) \big) = \sum_{n=0}^\infty 2^{-\nh} \cosi (\zp 2^n x),\quad x\in[0,1].
\ee
In this section we shall describe a dynamical system on $[0,1]^2$, alternatively $\Om = \{0,1\}^\N\times\{0,1\}^\N$, which produces the curve as its attractor.
For elements of $\Om$ we write for convenience $\om = ((\om_{-n})_{n\ge 0}, (\om_n)_{n\ge 1})$; one understands $\Omega$ as the space of $2$-dimensional sequences of Bernoulli random variables.  Denote by $\theta$ the canonical shift on $\Om$, given by
$$\theta:\Om\to\Om,\quad  \om\mapsto (\om_{n+1})_{n\in\Z}.$$
$\Om$ is endowed with the product $\sigma$-algebra, and the infinite product $\nu = \otimes_{n\in\Z} (\eh \delta_{\{0\}} + \eh \delta_{\{1\}})$ of Bernoulli measures on $\{0,1\}.$ We recall that $\theta$ is $\nu$-invariant.

Now let
$$T=(T_1,T_2):\Om \to [0,1]^2,\quad \om \mapsto \Big(\sum_{n=0}^\infty \om_{-n} 2^{-(n+1)}, \sum_{n=1}^\infty \om_n 2^{-n}\Big).$$ Let us denote by $T_1$ the first component of $T$, and by $T_2$ the second one.
It is well known that $\nu$ is mapped by the transformation $T$ to $\lambda^2$ (i.e. $\nu=\lambda^2\circ T$), the 2-dimensional Lebesgue measure. It is also well known that the inverse of $T$, the dyadic representation of the two components from $[0,1]^2$, is uniquely defined
apart from the dyadic pairs. For these we define the inverse to map to the sequences not finally containing only $0$. We now define $B:[0,1]^2\to [0,1]^2$, the so-called \emph{Baker transformation}, as
$$B =(B_1,B_2)= T\circ\theta \circ T^{-1}.$$
The $\theta$-invariance of $\nu$ directly translates into the $B$-invariance of $\lambda^2$:
$$\lambda^2 \circ B^{-1} = (\lambda^2\circ T) \circ \theta^{-1}\circ T^{-1} = (\nu\circ \theta^{-1}) \circ T^{-1} = \nu \circ T^{-1} = \lambda^2.$$
For $(\xi, x)\in[0,1]^2$ let us denote
$$T^{-1}(\xi, x) = \big((\bxi_{-n})_{n\ge 0}, (\bx_n)_{n\ge 1}\big).$$
Let us calculate the action of $B$ and its entire iterates on $[0,1]^2.$

\begin{lem}\label{l:action_B}
Let $(\xi, x) \in[0,1]^2$. Then for $k\ge 0$
$$
B^k(\xi, x)
=
\Big(2^k \xi (\mbox{mod } 1), \frac{\bxi_{-k+1}}{2} + \frac{\bxi_{-k+2}}{2^2}+\cdots+\frac{\bxi_0}{2^k} + \frac{x}{2^k}\Big),
$$
for $k\ge 1$
$$ B^{-k}(\xi, x)
=
\Big(\frac{\xi}{2^k}+ \frac{\bx_1}{2^k} + \frac{\bx_2}{2^{k-1}}+\cdots+ \frac{\bx_k}{2}, 2^k x (\mbox{mod } 1)\Big).
$$
\end{lem}
\begin{proof}
By definition of $\theta^k$ for $k\ge 0$ we have
$$B^k(\xi,x) = \Big(\sum_{n\ge 0} \bxi_{-n+k} 2^{-(n+1)}, \frac{\bxi_{-k+1}}{2} + \frac{\bxi_{-k+2}}{2^2}+\cdots+\frac{\bxi_0}{2^k} + \sum_{n\ge 1} \bx_n 2^{-(k+n)}\Big).$$
Now we can write
$$\sum_{n\ge 0} \bxi_{-n+k} 2^{-(n+1)} = 2^k \xi (\mbox{mod } 1)
\quad\textrm{and}\quad
\sum_{n\ge 1} \bx_n 2^{-(k+n)} = \frac{x}{2^k}.$$
This gives the first formula.
For the second, note that by definition of $\theta^{-k}$ for $k\ge 1$
$$B^{-k}(\xi, x) = \Big(\sum_{n\ge 0} \bxi_{-n} 2^{-(n+1+k)} + \frac{\bx_1}{2^k} + \frac{\bx_2}{2^{k-1}}+\cdots +\frac{\bx_k}{2}, \sum_{n\ge 1} \bx_{n+k} 2^{-n}\Big).$$
Again, we identify
$$\sum_{n\ge 1} \bx_{n+k} 2^{-n} = 2^k x (\mbox{mod } 1)
\qquad\textrm{and}\qquad
\sum_{n\ge 0} \bxi_{-n} 2^{-(n+1+k)} = \frac{\xi}{2^k}.$$
\end{proof}
For $k\in\Z, (\xi, x)\in[0,1]^2$ we abbreviate the $k$-th Baker transform of $(\xi,x)$ as
$$B^{k}(\xi, x) = \big( B^{k}_1(\xi, x), B^{k}_2(\xi, x) \big)  = (\xi_k, x_k),$$
where for $k\ge 0$
$$\xi_k = 2^k\xi (\mbox{mod } 1),
\quad\textrm{and}\quad
x_k = \frac{\bxi_{-k+1}}{2} + \frac{\bxi_{-k+2}}{2^2}+\cdots+\frac{\bxi_0}{2^k} + \frac{x}{2^k},$$
and for $k\ge 1$
$$\xi_{-k} = \frac{\xi}{2^k}+ \frac{\bx_1}{2^k} + \frac{\bx_2}{2^{k-1}}+\cdots + \frac{\bx_k}{2},
\quad\textrm{and}\quad x_{-k} = 2^k x (\mbox{mod } 1).$$
We will next interpret the Weierstrass curve $W$ by a transformation on our base space $[0,1]^2$.
Let
\bea
F: [0,1]^2\times\R^2 &\to& [0,1]^2\times \R^2,
\\
 (\xi,x, y_1, y_2)&\mapsto& \Big(B(\xi, x), 2^{-\eh} y_1 + \cos \big(\zp B_2(\xi, x)\big), 2^{-\eh} y_2 + \sin \big(\zp B_2(\xi,x)\big)\Big)
\\
&\mapsto& \Big(B(\xi, x), \gamma 
  y + \cosi \big(\zp B_2(\xi, x)\big) \Big),\qquad y=(y_1,y_2).
\eea

Here we denote $B = (B_1, B_2)$ for the two components of the Baker transform $B$. For convenience, we extend $W$ from $[0,1]$ to $[0,1]^2$ by setting
$$W(\xi, x) = W(x),\quad \xi, x\in[0,1].$$
We verify next that the graph of $W$ is an attractor for $F$. The skew-product structure of $F$ with respect to $B$ plays a crucial role. From it we see the alluded self-affine property.
\begin{lem}\label{l:attractor}
For any $\xi, x\in[0,1]$ we have
$$F\big(\xi, x, W(\xi, x)\big) = \Big(B(\xi, x), W\big(B(\xi, x)\big)\Big).$$
\end{lem}
\begin{proof}
We have by the $2\pi$-periodicity of the trigonometric functions
\bea
W\big(B_2(\xi,x)\big)
= W\Big(\frac{\bxi_0 + x}{2}\Big)
&=& \sum_{n=0}^\un 2^{-\nh} \cosi \Big(\zp 2^n \frac{\bxi_0+x}{2}\Big)\\
&=& \cosi \Big(\zp \frac{\bxi_0+x}{2}\Big) + \sum_{n=1}^\un 2^{-\nh} \cosi(\zp 2^{n-1} x)\\
&=& \cosi \Big(\zp \frac{\bxi_0+x}{2}\Big) + 2^{-\eh} \sum_{n=0}^\un 2^{-\frac{n}{2}} \cosi (\zp 2^n x)\\
&=& \cosi \big(\zp B_2(\xi, x)\big) + 2^{-\eh} W(x).
\eea
Hence by definition of $F$
$$\Big(B(\xi, x), W\big(B(\xi, x)\big)\Big)
= \Big(B(\xi, x), W\big(B_2(\xi, x)\big)\Big) = F\Big(\xi, x, W(\xi, x)\Big).$$
\end{proof}
Let us finally calculate the Jacobian of $F$ to gain insight on its stable manifolds. We obtain for $\xi, x\in[0,1], y_1, y_2\in\R$
\bea
DF(\xi, x, y_1, y_2) &=&
\left[
\begin{array}{cccc}
2&0&0&0\\
0&\eh&0&0\\
0& -\pi\sin \big(\zp  B_2(\xi, x) \big)& 2^{-\eh}&0\\
0& \phantom{-}\pi \cos \big(\zp  B_2(\xi, x) \big)&0&2^{-\eh}
\end{array}\right].
\eea
The Lyapunov exponents of the dynamical system associated with $F$ are given by $2, \eh$, and $\gamma:=2^{-\eh},$ the last being a double one.
The corresponding invariant vector fields are given by
$$\left(\begin{array}{c} 1\\0\\0\\0\end{array}\right),\ \  
X(\xi,x)
= \left(\begin{array}{c}0\\1\\  -\zp \sum_{n=1}^\un \gamma^n \sico \big(\zp B^n_2(\xi,x)\big)\end{array}\right),\ \ 
\left(\begin{array}{c} 0\\0\\1\\0\end{array}\right), \ \ 
\left(\begin{array}{c} 0\\0\\0\\1\end{array}\right),$$
as is straightforwardly verified.
Hence we have in particular for $\xi, x\in[0,1], y_1, y_2\in\R$
$$DF(\xi, x, y_1, y_2) X(\xi, x) = \frac{1}{2}\,\, X\big(B(\xi, x)\big).$$
Note that the vector $X$ spans an invariant stable manifold independent of $y_1, y_2$.
%

\section{Scaling properties}

We shall first establish an intrinsic link between the Weierstrass curve as the attractor of an underlying dynamical system and its stable manifold. This link gives rise to a kind of self similarity property, which subsequently allows us to study scaling properties of the curve. These will turn out crucial for the lower estimate on the Hausdorff dimension of its graph.

Let us first recall the measure supported by the stable manifold of our dynamical system, the so-called \emph{Sinai-Bowen-Ruelle measure} (SBR). Define $S(\cdot,\cdot)$ as the 3rd \& 4th component of $X$, namely
$$
S(\xi, x)
= - \zp \sum_{n=1}^\un \gamma^n \sico \big( \zp B^n_2(\xi,x) \big),\quad \xi, x\in[0,1],\ \gamma=2^{-\frac12},
$$
where $B^n_2$ is the second component of the $n$-th Baker transform.  It is noteworthy to see that in \eqref{eq:DefinitionW} the action in $x$ is to expand via the multiplicative power $2^n$, i.e. one sees $B_2^n(\xi,x)$ for some positive $n$. For the map $S(\xi,x)$ one sees that the action in $x$ is contracting! I.e.~the term appearing in $B_2^{-n}(\xi,x)$ for some positive $n$.

Let us calculate the action of $S$ on the $\lambda^2$-measure preserving map $B$. For $\xi, x\in[0,1]$  we have
\begin{align*}
S\big(B(\xi, x)\big)
&= - \zp \sum_{n=1}^\un \gamma^n \sico \Big(\zp B^n_2(B(\xi, x))\Big)
\\
&= - \zp \sum_{n=1}^\un \gamma^n \sico \big(\zp B_2^{n+1}(\xi, x)\big)
\\
&= - 2^{\eh} \times \zp  \sum_{k=1}^\un \gamma^k \left[\sico \big(\zp B^k_2(\xi, x)\big)\right] + \zp \sico \big(\zp B_2(\xi, x)\big)
\\
&= 2^{\eh} S(\xi, x) + \zp \sico\big(\zp B_2(\xi, x)\big).
\end{align*}
Finally, we define the Anosov skew product $G$ as 
\begin{align}
\nonumber
G: [0,1]^2\times\R^2 &\to [0,1]^2\times \R^2,
\\ \nonumber
 (\xi,x, v_1, v_2)&\mapsto
\Big(B(\xi, x), 2^{\eh} v_1 - \zp \sin \big(\zp B_2(\xi, x)\big), 2^{\eh} v_2 + \zp \cos \big(\zp B_2(\xi,x)\big)\Big)
\\
\label{eq:DefinitionG}
&\mapsto
\Big(B(\xi, x), 2^{\frac12} v + \zp \sico \big(\zp B_2(\xi, x)\big) \Big),\quad v=(v_1,v_2).
\end{align}

To summarize the above calculation we state the following result ({compare with Lemma \ref{l:attractor}}).
\begin{lem}\label{l:SBR}
For any $\xi, x\in[0,1]$ we define $G:[0,1]^2\to[0,1]^2\times \R^2$ as
$$G\big(\xi, x, S(\xi, x)\big)
= \Big(B(\xi, x), S\big(B(\xi, x)\big)\Big).$$
The push-forward measure of the Lebesgue measure in $\R^2$ to the graph of $S$ given by
$$\psi = \lambda^2\circ (\mbox{id}, S)^{-1}$$
on $\cB([0,1]^2)\otimes \cB(\R^2)$ is $G$-invariant.
\end{lem}
\begin{proof}
The first equation has been verified above. The $G$-invariance of $\psi$ is a direct consequence of the $B$-invariance of $\lambda^2$.
\end{proof}

For $x\in[0,1]$ let $\mu_x = \lambda\circ S(\cdot,x)^{-1}$. Then $\mu$, the measure on $[0,1]^2$ with marginals $\mu_x$, is the \emph{Sinai-Bowen-Ruelle measure} of $G$.

In the following key lemma we establish the link between $W$ and the stable manifold of $F$. For this purpose, we introduce the map $H:[0,1]^2\to \R^2$ as
$$
H(\xi,x)
= \sum_{n\in\Z} 2^{-\frac{n}{2}} \left[\cosi \big(\zp B^{-n}_2(\xi, x)\big) - \cosi \big(\zp B^{-n}_2(\xi, 0)\big)\right],
\quad \xi, x\in[0,1].$$

Then we have the following relationship between $W$ and $S$.

\begin{lem}\label{l:doubly_infinite_series}
For $x, y, \xi\in [0,1]$ we have
$$H(\xi, y) - H(\xi, x) = W(y) - W(x) - \int_x^y S(\xi, z) \dz.$$
\end{lem}
\begin{proof}
For $x, y, \xi\in[0,1]$ we have using the definition of $W$ in  \eqref{eq:DefinitionW}
\bea
H(\xi, y) - H(\xi, x)
&=& \sum_{n\in\Z} 2^{-\frac{n}{2}} \left[\cosi \big(\zp B^{-n}_2(\xi, y)\big) - \cosi \big(\zp B^{-n}_2(\xi, x)\big)\right]
\\
&=& \sum_{n=0}^\infty 2^{-\frac{n}{2}} \left[\cosi \big(\zp B^{-n}_2(\xi, y)\big) - \cosi \big(\zp B^{-n}_2(\xi, x)\big)\right]
\\
 &&\hspace{.5cm} + \sum_{k=1}^\infty 2^{\frac{k}{2}} \left[\cosi \big(\zp B^{k}_2(\xi, y)\big) - \cosi \big(\zp B^{k}_2(\xi, x)\big)\right]
\\
&=& W(y)-W(x) + \int_x^y \zp\, \sum_{k=1}^\infty 2^{-\frac{k}{2}} \sico \big(\zp B^k_2(\xi,z)\big) \dz
\\
&=& W(y) - W(x) - \int_x^y S(\xi, z) \dz.
\eea

This completes the proof.
\end{proof}

We will next assess the scaling properties of $H$. They will be crucial for the lower bound on the Hausdorff dimension of the graph of W.

\begin{pr}[Scaling of $H$]\label{p:scaling}
For $\xi, x\in[0,1]$, $\gamma=2^{-1/2}$, we have
$$H\big(B(\xi, x)\big) = \gamma H(\xi,x).$$
For $r>0$ define the set
$A_r = \{(\xi,x,y)\in[0,1]^3: |H(\xi, y) - H(\xi, x)|\le r\}.$
Then 
$$\lambda^3(A_{\gamma r}) = \gamma^2 \lambda^3(A_r).$$
\end{pr}
\begin{proof}
First note that by definition, setting $n-1 = k$, for $\xi, x\in[0,1]$
\bea
H(B(\xi, x))
&=&
\sum_{n\in\Z} 2^{-\frac{n}{2}} \left[\cosi \big(\zp B^{-n+1}_2(\xi, x)\big) - \cosi \big(\zp B^{-n+1}_2(\xi, 0)\big)\right]
\\
&=&
\gamma \sum_{k\in\Z} 2^{-\frac{k}{2}} \left[\cosi \big(\zp B^{-k}_2(\xi, x)\big) - \cosi \big(\zp B^{-k}_2(\xi, 0)\big)\right]
\\
&=& \gamma H(\xi, x).
\eea
For the second claim, note that the first one, $H(B(\xi, x)) = \gamma H(\xi,x)$,  gives
\bea
&&
\int_{[0,1]^3} \1_{[0,r]}\Big(\big|H(B(\xi, y)) - H(B(\xi, x))\big|\Big) \ud x \ud y \ud\xi\\
&&\hspace{1cm}
= \int_{[0,1]^3} \1_{[0,r]}\Big(\gamma\big|H(\xi, y) - H(\xi, x)\big|\Big) \ud x \ud y \ud \xi
= \lambda^3(A_{\gamma^{-1} r}).
\eea
On the other hand, using the definition of $B$, we may calculate
\bea
\lambda^3(A_{\gamma^{-1} r})
&=&
\int_{[0,1]^3} \1_{[0,r]}\Big(\big|H(B(\xi, y)) - H(B(\xi, x))\big|\Big) \ud x \ud y \ud \xi
\\
&=&
 \int_{[0,1]^3} \1_{[0,r]}\Big(\big|H(2\xi(\textrm{mod }1), \frac{\bxi_0+y}{2}) - H(2\xi(\textrm{mod }1), \frac{\bxi_0+x}{2})\big|\Big) \ud x \ud y \ud \xi
\\
&=& 
 \frac{1}{2} \int_{[0,1]^3} \1_{[0,r]}\Big(\big|H(2\xi(\textrm{mod }1), \frac{y}{2}) - H(2\xi(\textrm{mod }1), \frac{x}{2})\big|\Big) \ud x \ud y \ud \xi
\\
&&\hspace{0.5cm}
+\frac{1}{2} \int_{[0,1]^3} \1_{[0,r]}\Big(\big|H(2\xi(\textrm{mod }1), \frac{1+y}{2}) - H(2\xi(\textrm{mod }1), \frac{1+x}{2})\big|\Big) \ud x \ud y \ud \xi
\\
&=&
 (\eh+\eh)\, 2\, \int_{[0,1]^3} \1_{[0,r]}\Big(\big|H(\xi', y') - H(\xi', x')\big|\Big) \ud x' \ud y' \ud \xi'
\\
&=&
2 \lambda^3(A_r) = \frac1{\gamma^2}\lambda^3(A_r).
\eea
For obtaining the first equality in the last line, we set $\xi'= 2 \xi(\textrm{mod }1), x' = \frac{x}{2}, y' = \frac{y}{2}$ resp. $\xi'=2 \xi (\textrm{mod }1), x' = \frac{x+1}{2}, y' = \frac{y+1}{2}$; in either case $\ud x \ud y \ud \xi = 2 \ud x' \ud y' \ud \xi'$.

The combination of the two preceding equations yields
$$\gamma^{-2} \lambda^3(A_r) = \lambda^3(A_{\gamma^{-1} r}).$$
Replacing $r$ with $\gamma r$ and multiplying the equation by $\gamma^2$, we obtain the desired equation. 
\end{proof}

From the preceding scaling statement we can easily deduce the following practical corollary.

\begin{co}\label{c:scaling_inequality_H}
There are constants $c, C>0$ such that for any $r>0$ we have
$$
c\, r^2 \le \lambda^3\Big(\big\{(\xi, x, y)\in[0,1]^3: |H(\xi,y)-H(\xi,x)| < r\big\}\Big)\le C r^2.
$$
\end{co}
\begin{proof}
Iterating the last statement of the preceding Proposition, we get for any $n\in\N$
\bea
&&\lambda^3\Big(\big\{(\xi, x, y)\in[0,1]^3: |H(\xi,y)-H(\xi,x)| < \gamma^n\big\}\Big)\\
 &&\hspace{.5cm}= \gamma^{2n}\lambda^3\Big(\big\{(\xi, x, y)\in[0,1]^3: |H(\xi,y)-H(\xi,x)| < 1\big\}\Big).
 \eea
Choose $r>0$. We may assume $r<1$, since otherwise the claim is trivial. Next choose $l\in \N$ such that $\gamma^{l+1}\le r\le \gamma^l.$ Then
\bea
&&\lambda^3\Big(\big\{(\xi,x,y)\in[0,1]^3: |H(\xi,y)-H(\xi,x)| < r \big\}\Big)
\\
&&\hspace{1.5cm} \le \lambda^3\Big(\big\{(\xi,x,y)\in[0,1]^3: |H(\xi,y)-H(\xi,x)| < \gamma^l \big\} \Big)
\\
&&\hspace{1.5cm} = \gamma^{2l} \lambda^3\Big(\big\{(\xi,x,y)\in[0,1]^3: |H(\xi,y)-H(\xi,x)| < 1 \big\}\Big)
\\
&&\hspace{1.5cm} \le r^2 \gamma^{-2}\lambda^3\Big(\big\{(\xi,x,y)\in[0,1]^3: |H(\xi,y)-H(\xi,x)| < 1 \big\}\Big).
\eea
Hence by setting $C = \gamma^{-2} \lambda^3\big(\big\{(\xi,x,y)\in[0,1]^3: |H(\xi,y)-H(\xi,x)| < 1 \big\}\big)$, we get the right hand side of the claimed inequality.

A similar argument for the left hand side reveals that setting $c = \gamma^2 \lambda^3\big(\big\{(\xi,x,y)\in[0,1]^3: |H(\xi,y)-H(\xi,x)| < 1 \big\}\big)$ finishes the proof.
\end{proof}


\section{A lower bound for the Hausdorff dimension}

In this section we give a lower estimate for the Hausdorff dimension of the graph of the curve $W$. We follow arguments set out in Keller \cite{keller17-Publication-of-2015}. The role of his telescoping arguments will be taken by the scaling properties of the curve set out in the preceding section. The starting point in Keller's \cite{keller17-Publication-of-2015} analysis is the idea to estimate the local dimension of the Lebesgue measure on the graph by taking small neighborhoods of the curve $W$ defined through the stable fibers of the underlying flow. Surprisingly, these neighborhoods are closely linked to the increments of the functions $H$ studied in the preceding section. The choice of the neighborhoods provides at the same time a geometric interpretation of increments of $H$. Stable fibers are carried by the invariant stable vectors essentially given by $S$. For $\xi,x,v\in[0,1], w=(w_1, w_2)\in\R^2$ consider the solution curves having $S$ as tangent defined by the following initial value problem
\bea
l:[0,1]\to \R^2,\qquad
\frac{d}{dv} l_{(\xi,x,w)}(v) = S(\xi,v),
\quad \textrm{with}\quad
l_{(\xi,x,w)}(x) = w.
\eea
If we substitute $w$ by $W(x)$, then $l_{(\xi,x,W(x))}$ is a curve passing at $x$ through the graph of $W$ and moving along the stable fibers. We now look into the vertical distance between the strong stable fibers through the points $(\xi,y, W(y))$ and $(\xi,x, W(x))$. For $\xi, x, y\in[0,1]$, $x<y$ we have
\bea
l_{(\xi,y, W(y))}(y) - l_{(\xi, x, W(x))}(y)
&=& W(y) - \, l_{(\xi, x, W(x))}(y) \\
&=& W(y) - W(x) - \big(\, l_{(\xi, x, W(x))}(y) - l_{(\xi,x,W(x))}(x) \,\big)\\
&=& W(y)-W(x) - \int_x^{y}S(\xi, z) dz\\
&=& H(\xi, y) - H(\xi, x).
\eea
In other words, differences between fibers convert into differences between the increments of $H$.

We now define small neighborhoods of the stable fibers at points $(x,W(x))$ of the graph of $W$. They will be needed to determine the local dimension of the Lebesgue measure on the graph, and therefore to give a lower bound for its Hausdorff dimension. For $K>0, N\in\N, \xi, x\in[0,1]$ let $I_N(x)$ be a neighborhood of $x$ of diameter $ 2^{-N}$ with dyadic boundaries, and
\be\label{e:small_neighborhoods}
V_N(\xi, x) = \big\{(v,w)\in [0,1]\times\R^2: v\in I_N(x), |w-l_{(\xi, x, W(x))}(v)|\le K\cdot 2^{-N}\big\}.
\ee
Define the pushforward of the Lebesgue measure on the graph of $W$ (or the lift of the Lebesgue measure $\lambda$ on $[0,1]$ to the graph of $W$) by
$$m = \lambda\circ (\mbox{id}, W)^{-1}
\quad \textrm{on}\quad \cB([0,1])\otimes \cB(\R^2).$$

We will be interested in giving a lower estimate of the local dimension of $m$ at $(x, W(x))$ for $\xi, x\in[0,1]$, calculated by
\be\label{e:density_limit}
\liminf_{N\to\un} \frac{\log m\big(V_N(\xi, x)\big)}{\log 2^{-N}}.
\ee
That this bound does not depend on $\xi$ has also been seen in Keller \cite[Remark 3.5]{keller17-Publication-of-2015}. 

We first investigate how $m$ scales on $V_N(\xi, x).$ Recall the notation for the iterated Baker transform $B^N(\xi, x) = (\xi_N, x_N)$ for $N\in\Z$.
\begin{lem}\label{l:telescoping}
For $N\in\N$, $K>0$, $\xi,x\in[0,1]$ we have
\bea
m\Big(V_N(\xi, x)\Big) &=& \lambda\Big(\{v\in I_N(x): |W(v) - l_{(\xi, x, W(x))}(v)|\le K 2^{-N}\}\Big)\\
 \hspace{1cm} &=& 2^{-N} \lambda\Big(\{u\in[0,1]: |H(\xi_{-N},u) - H((\xi_{-N}, x_{-N}))|\le K\gamma^N\}\Big).
 \eea
\end{lem}
\begin{proof}
Since the neighborhood of $x$ of diameter $ 2^{-N}$ with dyadic boundaries, $I_N(x)$, has length $2^{-N}$ and in our model $[0,1]$ is topologically the unit circle, we have
\begin{eqnarray}\label{equation:scaling}
  \nonumber
&&\lambda\big(\{v\in I_N(x): |W(v) - l_{(\xi, x, W(x))}(v)|\le K 2^{-N}\}\big)
\\ \nonumber
&&\hspace{.5cm} = 2^{-N}\cdot \lambda\big(\{u\in [0,1]: |W(2^{-N}u) - l_{(\xi, x, W(x))}(2^{-N}u)|\le K 2^{-N}\}\big)
\\ \nonumber
&&\hspace{.5cm} = 2^{-N}\cdot \lambda\big(\{u\in [0,1]: |W(2^{-N}u) - W(x)
\\ \nonumber
&&\hspace{6.5cm} - \big(l_{(\xi, x, W(x))}(2^{-N}u)-l_{(\xi, x, W(x))}(x)\big)|\le K 2^{-N}\}\big)
\\ \nonumber
&&\hspace{.5cm} = 2^{-N}\cdot \lambda\big(\{u\in [0,1]: |W(2^{-N}u) - W(x) - \int_x^{2^{-N}u} S(\xi, t) \dt|\le K 2^{-N}\}\big)
\\ \label{eq:Telescopingtrick}
&&\hspace{.5cm} = 2^{-N}\cdot \lambda\big(\{u\in [0,1]: |W(B^N_2(\xi_{-N},u)) - W(B^N_2(\xi_{-N},x_{-N}))
\\ \nonumber
&&\hspace{6.5cm}- \int_{B_2^N(\xi_{-N}, x_{-N})}^{B_2^N(\xi_{-N},u)} S\big(B^{N}(B^{-N}(\xi,t))\big) \dt |\le K 2^{-N}\}\big)\\ \nonumber
&&\hspace{.5cm} = 2^{-N}\cdot \lambda\big(\{u\in [0,1]: |H(B^N(\xi_{-N},u)) - H(B^N(\xi_{-N},x_{-N}))|\le K 2^{-N}\}\big)\\ \nonumber
&&\hspace{.5cm} = 2^{-N}\cdot \lambda\big(\{u\in [0,1]: |H(\xi_{-N},u) - H(\xi_{-N},x_{-N})|\le K \gamma^N\}\big).
\end{eqnarray}

To obtain the last equation in (\ref{equation:scaling}), we apply $N$-times the first part of Proposition \ref{p:scaling}. This yields the desired equation.
\end{proof}
Lemma \ref{l:telescoping} allows a first specification of the local density limit of (\ref{e:density_limit}). In fact, for $N\in\N$ we have
\begin{align}
\label{e:density_limit_modification}
\nonumber
& \frac{\log m\big(V_N(\xi, x)\big)}{\log 2^{-N}}\\
&\hspace{1.5cm} = 1 + \frac{\log \lambda\Big(\Big\{u\in[0,1]: |H(\xi_{-N},u) - H(B^{-N}(\xi, x)|\le K\gamma^N\Big\}\Big)}{\log 2^{-N}}.\
\end{align}
Hence, we will have to obtain a lower estimate on
\begin{align}
\label{e:density_limit_2}
&\hspace{1cm} \liminf_{N\to\un} \frac{\log \lambda\Big(\Big\{u\in[0,1]: |H(\xi_{-N},u) - H(B^{-N}(\xi, x))|\le K\gamma^N\Big\}\Big)}{\log 2^{-N}}.
\end{align}

By means of Corollary \ref{c:scaling_inequality_H}, we can now give an auxiliary estimate leading to determine a lower bound on (\ref{e:density_limit_2}).

\begin{lem}\label{l:Borel-Cantelli}
There exists a constant $C$ such that for $0<r<\eh, N\in \N$
$$
\lambda^2\Big(\Big\{
(\xi,x)\in[0,1]^2: \lambda\big(\{y\in[0,1]: |H(\xi_{-N}, y) - H(B^{-N}(\xi, x))|\le r\}\big)\ge r^{2-\eta}
             \Big\}\Big)
\le C\, r^{\eta}.$$
\end{lem}
This is similar to the \emph{Marstrand projection estimate} of Keller \cite[Section 3.5]{keller17-Publication-of-2015}.
\begin{proof}
In fact, by the $B$-invariance of $\lambda^2$ and Corollary \ref{c:scaling_inequality_H} we may write, with a universal constant $C>0$
\bea
&&\lambda^2\Big(\Big\{ (\xi,x)\in[0,1]^2:
\lambda\big(\big\{y\in[0,1]: |H(\xi_{-N},y) - H(B^{-N}(\xi, x))|\le r\big\}\big) \ge r^{2-\eta}
           \Big\} \Big)
\\
\nn&&\hspace{.5cm}
= \lambda^2\Big(\Big\{ (\xi,x)\in[0,1]^2:
\lambda\big(\big\{y\in[0,1]: |H(\xi,y) - H(\xi, x)|\le r\big\}\big)\ge r^{2-\eta}
           \Big\}\Big)
\\
\nn&&\hspace{.5cm}
\leq  r^{-(2-\eta)} \lambda^3\Big(\Big\{ (\xi, x, y)\in[0,1]^3: |H(\xi,y) - H(\xi,x)|\le r\Big\}\Big)
\\
\nn&&\hspace{.5cm}
\le C r^{-(2-\eta)} r^{2}
\\
\nn&&\hspace{.5cm}
= C  r^\eta,
\eea
where the first domination follows from Markov's inequality.

This concludes the proof.
\end{proof}
It remains to apply Borel-Cantelli's lemma to obtain the lower bound on (\ref{e:density_limit_2}).

\begin{pr}\label{p:lower_bound}
We have
$$\liminf_{N\to\un} \frac{\log \lambda\Big(\Big\{y\in[0,1]: |H(\xi_{-N}, y ) - H(B^{-N}(\xi, x))|\le K\gamma^N\Big\}\Big)}{\log 2^{-N}} \ge 1.$$
\end{pr}
\begin{proof}
Applying the lemma of Borel-Cantelli with $r_N = K \gamma^{-N}$ to the result of Lemma \ref{l:Borel-Cantelli} we get for $\lambda^2-$a.e. $(\xi,x)\in[0,1]^2$ that
$$
\limsup_{N\to\un} \gamma^{-(2-\eta)} \lambda\Big(\Big\{y\in[0,1]: |H(\xi_{-N}, y) - H(B^{-N}(\xi, x))|\le K \gamma^N\Big\}\Big) \le K^{2-\eta}.
$$
This implies
\begin{align*}
\liminf_{N\to\un} & \frac{\log \lambda\Big(\Big\{y\in[0,1]: |H(\xi_{-N}, y) - H(B^{-N}(\xi, x))|\le K\gamma^N\Big\}\Big)}{\log 2^{-N}}
\\
& \hspace{8cm}
 \ge (2-\eta) \frac{\log \gamma}{\log \eh} = (2-\eta) \eh.
\end{align*}
Since $\eta>0$ is arbitrary, this implies the desired estimate.
\end{proof}
 %

We finally obtain a lower estimate for the Hausdorff dimension of the graph of $W$.
\begin{thm}
\label{t:lower_estimate}
Let $m = \lambda\circ (\mbox{id}, W)^{-1}$, $x\in[0,1],$ and for $N\in\N$ let $V_N(\xi,x)$ be defined by (\ref{e:small_neighborhoods}). Then
$$\liminf_{N\to \infty} \frac{\log m\Big(V_N(\xi,x)\Big)}{\log 2^{-N}} \ge 2.$$
The Hausdorff dimension of the graph of $W$ in \eqref{eq:DefinitionW} is bounded from below by $2.$
\end{thm}
\begin{proof}
Combine (\ref{e:density_limit_modification}) with the result of Proposition \ref{p:lower_bound}. For the claim on the Hausdorff dimension, consult the remark in Baranski et al.~\cite{baranski14published} after Definition 1.1.
\end{proof}

\section{Upper bound for the Hausdorff dimension}

In the previous section we computed a lower bound for the Hausdorff dimension of \eqref{eq:DefinitionW}. It easy to recall that $W$ in \eqref{eq:DefinitionW} is a $1/2$-H\"older continuous function. For $\alpha$-H\"older functions a general result exists stating that for a H\"older function $f:A\to\R^d$ ($A\subset \R$)
\be
\label{eq:UpperBoundForHolderFunc}
\dim \Big( \textrm{Graph}_f (A) \Big)
	:=
	\dim \Big( \big\{ (t,f(t)): t\in A \big\} \Big)
	\leq 1 + (1-\alpha) \Big(d \wedge \frac1\alpha\Big).
\ee
It can be shown that the above results cannot be improved under the H\"older condition alone. In the particular case of our map $W$ ($d=2$, $\alpha=1/2$), we have $\dim \big( \textrm{Graph}_W ([0,1]) \big)\leq 2 $. In \cite{baranski02} it is shown that the Box dimension of the graph of $W$ is indeed $3-2\alpha$, in particular for $\alpha=1/2$ the dimension is $2$ (see his Corollary 4.4). Recall that the Box dimension dominates the Hausdorff one.

Summing up, we state our result on the Hausdorff dimension of the graph of $W$.
\begin{thm}\label{t:main}
The Hausdorff dimension of the graph of $W$ is $2$.
\end{thm}
\begin{proof}
Combine the remarks on the upper bound with Theorem \ref{t:lower_estimate}.
\end{proof}




\begin{thebibliography}{99}
\bibitem{baranski02} K. Baranski \emph{On the complexification of the Weierstrass non-differentiable function.} Anales-Acadamiae Scientarium Fennicae Mathematica. Vol. 27 (2002), No. 2. Academia Scientarium Fennica.

\bibitem{baranski12} K. Baranski \emph{On the dimension of graphs of Weierstrass-type functions with rapidly growing frequencies.} Nonlinearity 25 (2012), no. 1, 193--209.

\bibitem{baranski14published} K. Baranski, B. Barany, J. Romanova. \emph{On the dimension of the graph of the classical Weierstrass function.}, Advances in Mathematics 265 (2014): 32-59.

\bibitem{carvalho2011} A. Carvalho. \emph{Hausdorff dimension of scale-sparse Weierstrass-type functions.}
Fund. Math. 213 (2011), no. 1, 1--13.

\bibitem{gubinelli2004} M. Gubinelli.
\emph{Controlling rough paths,} J. Funct. Anal. 216 (2004), no. 1, 86–140.

\bibitem{gubinelliimkellerperkowski2015}
M. Gubinelli, P. Imkeller, N. Perkowski. \emph{Paracontrolled distributions and singular PDEs.} Forum Math.Pi - Vol. 3 (2015), e6, 75.

\bibitem{gubinelliimkellerperkowski2016} M. Gubinelli, P. Imkeller, N. Perkowski.  \emph{A Fourier approach to pathwise stochastic integration.} Electron. J. Probab. 21 (2016), no. 2, 1-37.

\bibitem{hardy1916} G. H. Hardy. \emph{Weierstrass's non-differentiable function.} Trans. Amer. Math. Soc 17.3 (1916): 301-325.

\bibitem{hunt1998} B. Hunt. \emph{The Hausdorff dimension of graphs of Weierstrass functions.} Proc. Amer. Math. Soc., 126.3 (1998), 791--800.

\bibitem{keller17-Publication-of-2015}
G. Keller. \emph{A simpler proof for the dimension of the graph of the classical Weierstrass function.} Annales de l'Institut Henri Poincar\'e, Probabilit\'es et Statistiques. Vol. 53. No. 1. Institut Henri Poincar\'e, (2017).


\bibitem{tsujii01} M. Tsujii. \emph{Fat solenoidal attractors.} Nonlinearity 14 (2001), No. 5, 1011--1027.

\bibitem{imkellerproemel15} P. Imkeller, D. Pr\"omel. \emph{Existence of L\'evy's area and pathwise integration}, Communications on Stochastic Analysis, Vol. 9, No.1 (2015) 93--111.

\bibitem{mortersperes10} P. M\"orters, P. Peres. \emph{Brownian motion.} Vol. 30. Cambridge University Press, 2010.

\bibitem{shen2017-Publication-of-2015} W. Shen. \emph{Hausdorff dimension of the graphs of the classical Weierstrass functions.} Mathematische Zeitschrift (2017).


\bibitem{baranski15survey} K. Baranski. \emph{Dimension of the graphs of the Weierstrass-type functions.} Fractal Geometry and Stochastics V. Springer International Publishing, (2015). 77-91.
\end{thebibliography}
\end{document}